\documentclass[letterpaper, 11pt,  reqno]{amsart}

\usepackage{amsmath,amssymb,amscd,amsthm,amsxtra, esint}

\usepackage[implicit=true]{hyperref}

\usepackage[left=32mm, right=32mm, 
bottom=27mm]{geometry}

\allowdisplaybreaks[2]

\usepackage{color}

\definecolor{gr}{rgb}   {0.,   0.69,   0.23 }
\definecolor{bl}{rgb}   {0.,   0.5,   1. }
\definecolor{mg}{rgb}   {0.85,  0.,    0.85}
\definecolor{yl}{rgb}   {0.8,  0.7,   0.}
\definecolor{or}{rgb}  {0.7,0.2,0.2}

\newtheorem{theorem}{Theorem} [section]

\newtheorem{lemma}[theorem]{Lemma}
\newtheorem{proposition}[theorem]{Proposition}
\newtheorem{remark}[theorem]{Remark}

\newtheorem{definition}[theorem]{Definition}

\newtheorem{conjecture}{Conjecture}[section]

\DeclareMathOperator*{\intt}{\int}

\DeclareMathOperator*{\supp}{supp}

\newcommand{\noi}{\noindent}
\newcommand{\Z}{\mathbb{Z}}
\newcommand{\R}{\mathbb{R}}

\newcommand{\T}{\mathbb{T}}

\let\Re=\undefined\DeclareMathOperator*{\Re}{Re}

\newcommand{\F}{\mathcal{F}}

\newcommand{\dl}{\delta}

\newcommand{\eps}{\varepsilon}

\newcommand{\ld}{\lambda}

\newcommand{\ft}{\widehat}

\newcommand{\dx}{\partial_x}
\newcommand{\dt}{\partial_t}

\renewcommand{\o}{\omega}

\newcommand{\les}{\lesssim}
\newcommand{\ges}{\gtrsim}

\newcommand{\jb}[1]
{\langle #1 \rangle}

\newcommand{\ind}{\mathbf 1}

\newcommand{\M}{\mathcal{M}}
\def\e{\varepsilon}

\newcommand{\N}{\mathbb{N}}

\newcommand{\J}{\mathcal{J}}

\newcommand{\TT}{\mathcal{T}}

\newcommand{\BT}{{\bf T}}

\usepackage{tikz}

\usetikzlibrary{shapes.misc}
\usetikzlibrary{shapes.symbols}
\usetikzlibrary{shapes.geometric}
\tikzset{
	dot/.style={circle,fill=black,draw=black,inner sep=0pt,minimum size=0.5mm},
	>=stealth,
	}
\tikzset{
	ddot/.style={circle,fill=white,draw=black,inner sep=0pt,minimum size=0.8mm},
	>=stealth,
	}

\tikzset{decision/.style={ 
        draw,
        diamond,
        aspect=1.5
    }}

\tikzset{dia2/.style
={diamond,fill=white,draw=black,inner sep=0pt,minimum size=1mm},
	>=stealth,
	}

\tikzset{dia/.style
={star,fill=black,draw=black,inner sep=0pt,minimum size=1mm},
	>=stealth,
	}

\makeatletter
\def\DeclareSymbol#1#2#3{\expandafter\gdef\csname MH@symb@#1\endcsname{\tikz[baseline=#2,scale=0.15]{#3}}}
\def\<#1>{\csname MH@symb@#1\endcsname}
\makeatother

\DeclareSymbol{X}{-2.4}{\node[dot] {};}
\DeclareSymbol{1}{0}{\draw[white] (-.4,0) -- (.4,0); \draw (0,0)  -- (0,1.2) node[dot] {};}
\DeclareSymbol{2}{0}{\draw (-0.5,1.2) node[dot] {} -- (0,0) -- (0.5,1.2) node[dot] {};}

\DeclareSymbol{1}{-2.7}
 {
  \draw (0,0) node[dot]{};}

\DeclareSymbol{1'}{-2.7}
 {\draw (0,0) node[ddot]{};}

\DeclareSymbol{1''}{-2.7}
 {\draw (0,0) node[dia]{};}

\DeclareSymbol{3}{0}
 {\draw (0,0) node[dot]{} -- (0,1.2) node[ddot] {}; 
 \draw (-1.2,0) node[dot] {} -- (0,1.2)node[ddot] {} -- (1.2,0) node[dot] {};}

\DeclareSymbol{3'}{0}
 {\draw (0,0) node[dia]{} -- (0,1.2) node[ddot] {}; 
 \draw (-1.2,0) node[dia] {} -- (0,1.2)node[ddot] {} -- (1.2,0) node[dia] {};}

\DeclareSymbol{2'}{0}
 {\draw (0,0) node[dia]{} -- (0,1.2) node[ddot] {}; 
 \draw (-1.2,0) node[dia] {} -- (0,1.2)node[ddot] }

\DeclareSymbol{31}{-3}{
\draw (0,-1)node[dot] {} -- (0,0) node[ddot] {}-- (0.9, -1) node[dot] {}; 
\draw (0,-1)node[dot] {} -- (0,0) node[ddot] {}-- (-0.9, -1) node[dot] {}; 
\draw (0,0)node[ddot] {} -- (1.3,1) node[ddot] {}-- (1.3, 0) node[dot] {}; 
\draw  (1.3,1) node[ddot] {}-- (2.6, 0) node[dot] {}; 
}

\DeclareSymbol{32}{-3}{
\draw (0,-1)node[dot] {} -- (0,0) node[ddot] {}-- (0.9, -1) node[dot] {}; 
\draw (0,-1)node[dot] {} -- (0,0) node[ddot] {}-- (-0.9, -1) node[dot] {}; 
\draw (0,0)node[ddot] {} -- (0,1) node[ddot] {}-- (0.9, 0) node[dot] {}; 
\draw (0,0)node[ddot] {} -- (0,1) node[ddot] {}-- (-0.9, 0) node[dot] {}; 
}

\DeclareSymbol{33}{-3}{
\draw (2.6,-1)node[dot] {} -- (2.6,0) node[ddot] {}-- (3.5, -1) node[dot] {}; 
\draw (2.6,-1)node[dot] {} -- (2.6,0) node[ddot] {}-- (1.7, -1) node[dot] {}; 
\draw (0,0)node[dot] {} -- (1.3,1) node[ddot] {}-- (1.3, 0) node[dot] {}; 
\draw  (1.3,1) node[ddot] {}-- (2.6, 0) node[ddot] {}; 
}

\DeclareSymbol{3''}{0}
 {\draw (0,0) node[dia]{} -- (0,1.2) node[ddot] {}; 
 \draw (-1.2,0) node[dot] {} -- (0,1.2)node[ddot] {} -- (1.2,0) node[dia] {};}

\DeclareSymbol{3'''}{0}
 {\draw (0,0) node[dot]{} -- (0,1.2) node[ddot] {}; 
 \draw (-1.2,0) node[dot] {} -- (0,1.2)node[ddot] {} -- (1.2,0) node[dia] {};}

\DeclareSymbol{31'}{-3}{
\draw (0,-1.2)node[dia] {} -- (0,0) node[ddot] {}-- (0.9, -1.2) node[dia] {}; 
\draw (0,-1.2)node[dia] {} -- (0,0) node[ddot] {}-- (-0.9, -1.2) node[dia] {}; 
\draw (0,0)node[ddot] {} -- (1.3,1.2) node[ddot] {}-- (1.3, 0) node[dia] {}; 
\draw  (1.3,1.2) node[ddot] {}-- (2.6, 0) node[dia] {}; 
}

\DeclareSymbol{32'}{-3}{
\draw (0,-1.2)node[dia] {} -- (0,0) node[ddot] {}-- (0.9, -1.2) node[dia] {}; 
\draw (0,-1.2)node[dia] {} -- (0,0) node[ddot] {}-- (-0.9, -1.2) node[dia] {}; 
\draw (0,0)node[ddot] {} -- (0,1.2) node[ddot] {}-- (0.9, 0) node[dia] {}; 
\draw (0,0)node[ddot] {} -- (0,1.2) node[ddot] {}-- (-0.9, 0) node[dot] {}; 
}

\DeclareSymbol{33'}{-3}{
\draw (2.6,-1.2)node[dia] {} -- (2.6,0) node[ddot] {}-- (3.5, -1.2) node[dia] {}; 
\draw (2.6,-1.2)node[dia] {} -- (2.6,0) node[ddot] {}-- (1.7, -1.2) node[dia] {}; 
\draw (0,0)node[dot] {} -- (1.3,1.2) node[ddot] {}-- (1.3, 0) node[dot] {}; 
\draw  (1.3,1.2) node[ddot] {}-- (2.6, 0) node[ddot] {}; 
}

\newtheorem*{ackno}{Acknowledgments}

\numberwithin{equation}{section}
\numberwithin{theorem}{section}

\begin{document}

\title[Norm inflation for dNLS]{Norm inflation for the derivative nonlinear Schr\"odinger equation}

\author[Y.~Wang, and Y.~Zine]
{Yuzhao Wang and Younes Zine}

\address{
Yuzhao Wang\\
School of Mathematics, 
University of Birmingham, 
Watson Building, 
Edgbaston, 
Birmingham\\
B15 2TT, 
United Kingdom}

\email{y.wang.14@bham.ac.uk}

\address{
Younes Zine,  School of Mathematics\\
The University of Edinburgh\\
and The Maxwell Institute for the Mathematical Sciences\\
James Clerk Maxwell Building\\
The King's Buildings\\
Peter Guthrie Tait Road\\
Edinburgh\\ 
EH9 3FD\\
United Kingdom}

\email{y.p.zine@sms.ed.ac.uk}

%
%

%

%
%
\subjclass[2020]{35Q55, 35R25}

\keywords{derivative nonlinear Schr\"odinger equation; ill-posedness; norm inflation}

\begin{abstract}
In this note, we study the ill-posedness problem for
the derivative nonlinear Schr\"odinger equation (DNLS)
in the one-dimensional setting.
More precisely, by using a ternary-quinary tree expansion of the Duhamel formula we prove norm inflation in Sobolev spaces below the (scaling) critical regularity for the gauged DNLS. 
This ill-posedness result is sharp since DNLS is known to be globally well-posed in $L^2(\R)$ \cite{HGKNV}.
The main novelty of our approach is to control the derivative loss from the cubic nonlinearity 
by the quintic nonlinearity with carefully chosen initial data.
\end{abstract}

\maketitle


\section{Setup of the problem}
\medskip
\subsection{The derivative nonlinear Schr\"odinger equation}
We consider the derivative nonlinear Schr\"odinger equation (DNLS)
defined on $\R$: 
\begin{align}
\begin{cases}
 i\dt u +    \dx^2 u   =  i   \dx ( |u |^2u ) \\
u|_{t = 0} = u_0, 
\end{cases}
\ (x, t) \in \R^2, 
\label{dNLS}
\end{align}

\noi
where $ u = u(t,x) $ is a complex-valued function.
The equation \eqref{dNLS} was derived in the plasma physics literature \cite{MOMT}  and has been extensively studied from a theoretical perspective. 
It is known that \eqref{dNLS} is completely integrable and admits infinitely 
many conservation laws \cite{KN}.
See \cite{MO,KNV} and reference therein for recent developments.

%

If $u(x,t)$ solves \eqref{dNLS} on $\R$, 
then, for any $\ld >0$, 
\[u^\ld: = \lambda^{\frac12} u(\lambda^2 t, \lambda x)\] 
also solves \eqref{dNLS} with scaled initial data 
$\phi^\ld : = \lambda^{\frac12} \phi(\lambda x)$.
This scaling invariance heuristically suggests that the critical Sobolev regularity 
of DNLS \eqref{GdNLS} 
is given by $s_\textup{crit} : = 0$.

Therefore, it is natural to conjecture the following:

\begin{conjecture}
\label{CON:con1}
The DNLS equation  \eqref{dNLS} 
is well-posed in $H^s(\R)$ for $s \ge 0$,
and ill-posed for $s < 0$.
\end{conjecture}

Regarding well-posedness, Conjecture \ref{CON:con1} has seen some recent progress culminating in the breakthrough work \cite{HGKNV} which proves global well-posedness of \eqref{dNLS} in $L^2(\R)$.  We also mention the following works on the well-posedness theory for \eqref{dNLS} \cite{TF80,TF81,Herr2,Takaoka1, KNV, HGKV, HGKNV}. On the other hand, Biagioni and Linares showed a mild form of ill-posedness for \eqref{dNLS}: they showed that the data-to-solution map fails to be uniformly continuous (strictly) below $H^{\frac12}(\R)$.

In the study of \eqref{dNLS},
the gauge transform plays an important role.
Define the nonlinear map $\mathcal G: L^2(\R) \mapsto L^2 (\R)$ by
\begin{align}
\mathcal G f(x): = e^{-i \int_{-\infty}^x |f(y)|^2 dy} f(x).
\label{gauge}
\end{align}

\noi
Then, a smooth function $u$ satisfies \eqref{dNLS} if and only if $v = \mathcal{G}(u)$ satisfies
\begin{align}
\begin{cases}
i \dt v + \dx ^2 v = -i v^2 \dx\overline{ v} -\frac{1}{2} |v|^4 v \\
v|_{t = 0} = \phi, 
\end{cases}
\label{GdNLS}
\end{align}

\noi
with $\phi = \mathcal{G}(u_0)$.

Note that if $u$ solves \eqref{dNLS} in $L^2(\R)$ then $v = \mathcal{G}(u)$ solves \eqref{GdNLS} in $L^2(\R)$. Hence, we can transfer any well-posedness result for \eqref{dNLS} in $L^2(\R)$ to a well-posedness result for \eqref{GdNLS} in $L^2(\R)$, and vice-versa.

The aim of this note is to study the ill-posedness problem of  \eqref{GdNLS}
in Sobolev spaces $H^s$. One way of showing ill-posedness is to show the discontinuity of the solution map, 
which can be done through {\it norm inflation}. More precisely, given $s<0$, we say that \eqref{GdNLS} exhibits norm inflation in $H^s(\R)$ if for any $\eps >0$ and $\phi \in H^s (\R)$, there exists a solution $v$ to \eqref{GdNLS}
on $\R$ and $t\in (0,\e)$ such that

\noi
\begin{align}
\label{norm-inflation}
\|v(0) - \phi \|_{H^s (\R)} < \e \, \textup{ and } \, \|v(t) \|_{H^s (\R)} > \e^{-1} . 
\end{align}

\noi
Note that if \eqref{GdNLS} exhibits norm inflation in $H^s(\R)$, then it not possible to define a continuous data-to-solution map. We invite the reader to consult \cite{CCT1,CCT2,CCT2b,CK,CP,FO,IO,Kishimoto,OhWang,Oh2,Okamoto,Xia, Chevyrev, COW} and
references therein for more information on the norm inflation phenomena.

The main purpose of this note is to show norm inflation of the gauged DNLS equation
\eqref{GdNLS}. Our main result is the following:

\begin{theorem}
\label{THM:ill1}
Suppose $s < 0$.
Fix $\phi\in H^s (\R)$.
Then, given any $\e >0$, 
there exist a global solution $v_\e$ to the gauged equations \eqref{GdNLS} on $\R$ and $t\in (0,\e)$
such that \eqref{norm-inflation} holds.
\end{theorem}



%

%

To the best of our best knowledge,
Theorem \ref{THM:ill1}
is the first ill-posedness result of \eqref{GdNLS} in terms of discontinuity of the flow map.
We prove Theorem \ref{THM:ill1} via a Fourier analytic method. In \cite{IO} Iwabuchi-Ogawa, showed norm inflation \eqref{norm-inflation} 
with $u_0 =0$ for quadratic nonlinear Schr\"odinger equations. Building upon the work of Bejenaru and Tao \cite{BT}, their approach is based on based on a Picard iteration scheme to show norm inflation using the high-to-low energy transfer in the first Picard iterate. It turns out that this method is widely applicable and works particularly well with power type nonlinearities  \cite{CP,FO,Kishimoto,OhWang,Oh2,Okamoto}; yielding norm inflation for almost optimal Sobolev exponents in most cases.
Oh \cite{Oh2} removed the constrain $u_0=0$ 
and proved norm inflation with general initial data for the cubic nonlinear Schr\"odinger equation by introducing a ternary tree expansion of the Duhamel formula.

The main difficulty when one tries to apply these methods to \eqref{GdNLS} comes from the presence of derivative in \eqref{GdNLS}.
Okamoto \cite{Okamoto} pointed out that Iwabuchi-Ogawa's argument
is applicable for some dispersive equations where a derivative appears in the nonlinearity.
In particular, he proved norm inflation for the Kawahara equation:
\begin{align*}
\partial_t u - \partial_x^5 u + \partial_x(u^2) = 0,
\end{align*}
in $H^s(\R)$ with $s<-2$,
which is sharp in the sense that the well-posedness is known in $H^s$ when $s > -2$.

In \cite{Okamoto}, the strong dispersion of the Kawahara equation 
plays a crucial rule in absorbing the 
derivative loss along the Picard iteration.
Unfortunately, the dispersion of \eqref{GdNLS} is 
not strong enough to handle the derivative loss.
In order to achieve norm inflation below the critical scaling exponent $s<0$, we need to further exploit the cubic-quintic structure of the nonlinearity.
Since the derivative loss in \eqref{GdNLS} only appears in the cubic part of the nonlinearity,
we observe that by carefully choosing our initial data,
the cubic part can be controlled by the quintic component of the nonlinearity.
Therefore, the problem then essentially reduces to showing norm inflation for the quintic nonlinear Schr\"odinger equation below $L^2(\R)$, which is already known in the literature, see \cite{Kishimoto}.

Another difficulty comes from the nonlinearity, 
which consists of a derivative cubic term 
and a non-derivative quartic term.
This non-homogeneous structure makes the series expansion of a solution complicated, see \eqref{homogen} below. 
To track these multilinear expressions, we introduce a notion of ternary-quinary trees
generalising the ternary trees in \cite{Oh2}. We note that Kishimoto \cite{Kishimoto} had to deal with this issue with a polynomial nonlinearity (consisting of several terms) without any derivative loss.

\noi
\subsection{Further remarks}

We conclude this section with several remarks.

\noi
\begin{remark}\rm \label{REM:sharp}
Since \eqref{dNLS} is globally well-posed in $L^2(\R)$ \cite{HGKNV}, norm inflation cannot occur in $H^s(\R)$ for $s \ge 0$ in view of the continuity of the Gauge transform \eqref{gauge} on smooth functions (i.e. from $C\big( [0,T] ;L^2(\R)\big)$ to itself, for any $T>0$, see \cite{HerrPHD} for a proof in the periodic setting). Therefore, Theorem \ref{THM:ill1} is sharp up to the endpoint $s=0$.  
\end{remark}

\noi

\begin{remark}\rm
Regarding the periodic setting, 
Theorem \ref{THM:ill1} still holds,
i.e. we can prove norm inflation for the gauged DNLS on $\T$ in $H^s(\T)$ for $s <0$ (note that the gauged equation \eqref{GdNLS} is modified on the torus, see \cite{Herr2, DNY1}.
The proof is a minor modification of the real line case. We also note that in the setting of Fourier-Lebesgue spaces, \eqref{dNLS} was shown to be locally well-posed in \cite{DNY1} in the whole subcritical regime (i.e. in Fourier-Lebesgue spaces that scale like $H^s(\T)$, for any $s >0$).
\end{remark}

\noi
\begin{remark}\rm
Using standard norm inflation techniques, we can also prove norm inflation for \eqref{dNLS} in $H^s(\M)$ for $\M = \R$ or $\T$ for $s < -1$. The $-1 \le s < 0$ case in Conjecture \ref{CON:con1} remains however open.
\end{remark}

\noi
\begin{remark}\rm
On the torus, another notion of probabilisitic criticality was developed in \cite{DNY2} giving rise to a probabilistic critical exponent $s_p$. More precisely, let $\{g_n\}_{n \in \Z}$ be an i.i.d. family of standard complex Gaussian random variables and define for any $s \in \R$, the function

\noi
\begin{align*}
u^s_0(\o) := \sum_{n \in \Z} \frac{g_n (\o)}{\jb{n}^{s + \frac12}} e^{i nx}.
\end{align*}

\noi
Then the exponent $s_p$ is defined to be the smallest $s \in \R$ such that the second Picard iterate of \eqref{dNLS} (after some appropriate frequency truncation) with initial data given by $u_0^s$ stays bounded in $H^s(\T)$; see \cite{DNY2} for more details. A computation shows that for \eqref{dNLS}, we have $s_p = s_c = 0$. Hence, it is unlikely that any random data theory (see for instance \cite{BO96, DNY2, NORS}) would allow us to go beyond the $L^2(\T)$ well-posedness threshold.
\end{remark}

The rest of the paper is organized as follows.
We introduce the notion of ternary-quinary trees
and establish multilinear estimates in Section 2.
In Section 3, we prove Theorem \ref{THM:ill1}.

\section{Preliminary analysis}
\label{SEC:2}

\subsection{Power series expansion indexed by trees}

We define two Duhamel integral operators
$\J$ and $\mathcal{K}$ by 

\noi
\begin{align}
\label{Duhamel1}
\begin{split}
\mathcal J[v_1,v_2,v_3](t) & := -i
\int_0^t S(t-t') v_1(t')v_2(t')  \dx \overline{v }_3(t')dt', \\
\mathcal{K}[v_1,v_2,v_3,v_4,v_5](t) & := -\frac{1}{2} \int_0^t  S(t-t') v_1(t') \overline{v}_2(t')v_3(t')\overline{v}_4(t')v_5(t')dt',
\end{split}
\end{align}

\noi
where  $S(t):=e^{it \dx ^2} $ denotes the linear propagator associated with (\ref{GdNLS}).
We also use the following shorthand notations:

\noi
\begin{align}\label{ternary_quinary_op}
\begin{split}
 \J^3[v] & := \J[v,v,v], \\
 \mathcal{K}^5[v] &:= \mathcal{K}[v,v,v,v,v].
\end{split}
\end{align}

\noi
In what follows, a solution $v$ to \eqref{GdNLS} with $v|_{t = 0} = \phi$ is a distribution which satisfies the Duhamel formulation

\noi
\begin{align}
\label{Duhamel}
v(t) = S(t)\phi +  \J^3[v](t) + \mathcal{K}^5[v](t).
\end{align}

\noi
It is expected that
the following Picard iteration
\begin{align}
\label{Picard}
P_0 (\phi) = S(t) \phi \,\, \textup{ and } \,\, P_j (\phi) = S(t) \phi + \J^3[P_{j-1} (\phi)] + \mathcal{K}^5[P_{j-1} (\phi)], \,\, j \in \N,
\end{align}

\noi
converges to a solution $v$ to \eqref{GdNLS} at least for short times.
If this is the case, 
then $v$ can be expressed as a series of multilinear terms in $\phi$:
\begin{align}
\label{power}
v (t) = \sum_{j\in 3 (\N \cup \{0\}) + 5 (\N \cup \{0\})}^\infty H_j [\phi]  (t),
\end{align}

\noi
where $H_j [\phi] $ consists of all homogeneous multilinear terms in $\phi$ of degree $j$.
For instance,
\begin{align}
\label{homogen}
\begin{split}
H_0 [\phi] (t)& : = S(t) \phi; \\
H_3 [\phi]  (t) & : = \J^3[S(t) \phi]; \\
H_5 [\phi]   (t)& : = \mathcal{K}^5[S(t) \phi]; \\
H_6 [\phi]   (t)& : = \mathcal{K}^3 \big[ \J^3[S(t) \phi] \big]; \\
H_8 [\phi]   (t)& : = \mathcal{K}^5 \big[ \J^3[S(t) \phi]\big] +  \J^3 \big[ \mathcal K^5[S(t) \phi]\big]; \\
& \cdots\cdots
\end{split}
\end{align}

\noi
Since the nonlinearity of \eqref{GdNLS} 
consists of both cubic and quintic nonlinearities,
the series expansion \eqref{power} and \eqref{homogen} gets more involved as $j$ gets larger.
Inspired by \cite{Chris,Oh2},
we shall use trees to index the series \eqref{power}.
To this purpose, we introduce the following notion of ternary-quinary trees: 

\begin{definition} \label{DEF:tree} \rm
(i) Given a partially ordered set $\TT$ with partial order $\leq$, 
we say that $b \in \TT$ 
with $b \leq a$ and $b \ne a$
is a child of $a \in \TT$,
if  $b\leq c \leq a$ implies
either $c = a$ or $c = b$.
If the latter condition holds, we also say that $a$ is the parent of $b$.

\smallskip 

\noi
(ii) 
A tree $\TT$ is a finite partially ordered set,  satisfying the following properties:
\begin{itemize}
\item Let $a_1, a_2, a_3, a_4 \in \TT$.
If $a_4 \leq a_2 \leq a_1$ and  
$a_4 \leq a_3 \leq a_1$, then we have $a_2\leq a_3$ or $a_3 \leq a_2$,

\item
A node $a\in \TT$ is called terminal, if it has no child.
A non-terminal node $a\in \TT$ is a node 
with  exactly {\it three} or {\it five} children,

\item There exists a maximal element $r \in \TT$ (called the root node) such that $a \leq r$ for all $a \in \TT$,

\item $\TT$ consists of the disjoint union of $\TT^0$ and $\TT^\infty$,
where $\TT^0$ and $\TT^\infty$
denote  the collections of non-terminal nodes and terminal nodes, respectively.
\end{itemize}
Given a tree $ \TT $, we denote by $\mathfrak n_3(\TT) $ (resp. $\mathfrak n_5(\TT) $) the number of non-terminal nodes which have three (resp. five) children.
\end{definition}

We also denote the collection of trees in the $(k,p)$-th generation (i.e. 
with $k$ parental nodes with three children and $p$ parental nodes with five children) by $\BT^{3,5}(k,p)$:
\begin{align}
\label{T35}
 \BT^{3,5}(k,p) & := \big\{ \TT :  \TT \text{ is a tree with } (\mathfrak n_3(\TT), \mathfrak n_5(\TT) )=(k,p) \}.
\end{align}

\noi
Note that the number $|\TT|$  of nodes in a tree $\TT \in \BT^{3,5} (k,p)$ is
$3k+5p+1$ for $k,p\in \N \cup \{0\}$.
In particular, the number of non-terminal nodes is $|\TT^0| = k+p$ and 
the number of terminal nodes is $|\TT^\infty| = 2k+4p+1$.

\begin{remark}\rm
The ternary-quinary trees defined in Definition \ref{DEF:tree}
generalise the ternary trees introduced in \cite{Oh2} 
by allowing some parental nodes 
to have five children.
It is easy to see that $\BT^{3,5} (k,0)$ and $\BT^{3,5} (0,p)$ 
consist of only ternary trees
and quinary trees respectively. 
\end{remark}

\begin{remark}\rm
\label{REM:def}
Let $\TT = \{ a \}_{a \in \TT} $ be a tree. Then, for every $a \in \TT$, we denote by $\TT_a$ the sub-tree whose root node is $a$. $\TT^0_a$ and $\TT^{\infty}_a$ are the associated sets of non-terminal and terminal nodes respectively.
Let $r\in \TT$ be the root node of $\TT$, then
\[
\TT = \TT_r.
\]
And similarly, $\TT^0 = \TT^0_r$ and $\TT^\infty = \TT^\infty_r$.
\end{remark}

We have the following bound
on the number of trees in $\BT^{3,5} (k,p)$:
\begin{lemma}\label{LEM:tree}
Let $\BT^{3,5}(k,p)  $ be as in \eqref{T35}.
Then, there exists $C>0$ such that 
\begin{align*}
\# \BT^{3,5}(k,p) & \leq C^{k+p},
\end{align*}
\noi
for all $(k,p) \in (\N \cup \{0\})^2$.
\label{BoundNumberOfTrees}
\end{lemma}

\begin{proof} 
We observe that 
\[
\# \BT^{3,5}(k,p) \le \# \BT^{3,5}(0,k+p),
\]

\noi
where $\BT^{3,5}(0,k+p)$ consists of all quinary trees in the $(k+p)$-th generation.
Then the bound follows from the same argument as Lemma 2.3 in \cite{Oh2}.
\end{proof}

Next, we associate operators to trees in the following manner.
Fix $\phi $ a function.
Let $\TT \in \BT^{3,5}(k,p)$ for
$(k,p) \in \N^2 \cup\{(0,0)\}$ be a tree.
Given functions $\phi_1,\cdots,\phi_{2k+4p+1}$,
we formally associate $\Psi (\TT, \phi_1,\cdots,\phi_{2k+4p+1})$, a multilinear operator, by the following rules:

\noi
\begin{itemize}
\item Replace a non-terminal node by the Duhamel integral operator $\J$ (resp. $\mathcal{K}$) defined in \eqref{Duhamel1}
with its three (resp. five) children as arguments $u_1, u_2$ and $ u_3$ (resp.  $u_1, u_2, u_3, u_4$ and $ u_5$).
\item Replace a terminal node by the linear solution $S(t) \phi_j$, $j  =1, \cdots, 2k+4p+1$. 
\end{itemize}

\noi
In the following, we set $\Psi_\phi(\TT) = \Psi(\TT, \phi,\cdots,\phi)$.
Therefore, $\Psi_\phi$ denotes a mapping
from $ \bigcup_{k,p \geq 0}\BT^{3,5}(k,p) $ to  $\mathcal{D}'(\M\times (-T, T))$. 
Note that, if  $\TT \in \BT^{3,5}(k,p)$, 
then $\Psi_\phi(\TT) $ is  $(2k+4p+1)$-linear in $\phi$. At times, we might identify the multilinear expression $\Psi_{\phi}(\TT)$ with its associated tree $\TT$ when the base function $\phi$ is fixed.

For $ j \geq 0 $, we define $ \Xi_{0,0}(t) = S(t) \phi$ and 
\begin{align}
\Xi_{k,p}^{3,5} (\phi)
& : =  \sum_{ \substack{ \TT \in \BT^{3,5}(k,p) } } \Psi_\phi (\TT),
\label{SumTree}
 \end{align}
for $k+p \geq 1$.
Finally,
we can rewrite \eqref{power} the series expansion of the solution $v$
to \eqref{Duhamel} 
as:
\begin{align}
\label{power1}
v = \sum_{k,p\ge 0} \Xi_{k,p}^{3,5} (\phi) ,
\end{align}

\noi
where $ \Xi_{k,p}^{3,5} (\phi)$ consists of homogeneous multilinear terms in $\phi$ of degree $2k+5p+1$.
We group all the $j$-th generation of the Picard iterations in $\Xi_j (\phi)$, i.e. 
\begin{align}
\label{Picard_j}
\Xi_j (\phi) : = \sum_{j = k+p} \Xi_{k,p}^{3,5} (\phi) ,
\end{align}

\noi
Then \eqref{power1} can be further written as
\begin{align}
\label{power2}
v = \sum_{j=0}^\infty \Xi_j  (\phi).
\end{align}

\noi
In the following, we shall show the convergence of the series \eqref{power2}
with some specific initial data $\phi$.

\medskip
\subsection{Multilinear estimates}\label{SUBSEC:multi}
In this subsection, we establish some multilinear estimates
with special initial data.
Fix $N  \gg 1$ (to be chosen later).
We define $\phi$ by setting
\begin{align}
\ft \phi (\xi)= R\big\{\ind_{2N+ Q_A} (\xi)+ \ind_{3N+ Q_A}(\xi)\big\}, 
\label{phi10}
\end{align}
	
\noi
where $Q_A = \big[-\frac A2, \frac A2\big)$, 
$R = R(N)   $ is a real parameter,  and $A = A(N)\gg 1$, satisfying
$A\ll N$, is to be chosen later.
\label{phi1a}
\noi
Note that we have
\begin{align}
\| \phi\|_{H^s} \sim N^s R A^\frac{1}{2} \,\, \textup{ and } \,\, \|\phi\|_{\F L^1} \sim R A
\label{data}
\end{align}

\noi
for any $s \in \R$.

Now we are ready to state our key multilinear estimates.

\noi
\begin{lemma}\label{LEM:nonlin1} Let
$\phi$ be as in \eqref{phi10}. We have for $t \geq 0$,

\noi
\begin{align}
\| \Xi_{k,p}^{3,5} (\phi) (t) \|_{\F L^1}
& \leq (Ct)^{k+p} N^{k}  (RA)^{2k + 4p + 1},
\label{nonlin1a}\\
\| \Xi_{k,p}^{3,5} (\phi) (t)  \|_{\F L^\infty}
& \leq (Ct)^{k+p} N^k (RA)^{2k + 4p} R,
\label{nonlin1b}\\
\big\| \partial_x \big(\Xi_{k,p}^{3,5} (\phi)\big) (t)  \big\|_{\F L^{\infty}}
& \leq (Ct)^{k+p} N^{k+1} (RA)^{2k + 4p} R,
\label{nonlin1c}
\end{align}
for all $(k,p) \in (\N \cup \{0\})^2$ and $\TT \in \BT^{3,5}(k,p)$.
\end{lemma}

\noi
\begin{proof} 
By Lemma \ref{LEM:tree} and \eqref{SumTree}, 
we only need to prove \eqref{nonlin1a}, \eqref{nonlin1b}, and  \eqref{nonlin1c} 
with $\Xi_{k,p}^{3,5} (\phi)$ replaced by $\Psi_\phi (\TT) $. 

To prove \eqref{nonlin1a}, 
we proceed by induction on 
$n = | \TT | $. If $|\TT| = 1$, i.e. the tree $\TT$ only has a single node
(and thus $\Psi_\phi (\TT) (t)=  \Xi_{0,0}(t) = S(t) \phi$), then 
\eqref{nonlin1a} follows from \eqref{data}. 
Fix $n \geq 1 $ and assume that 
\eqref{nonlin1a} holds for all trees $\TT \in \BT^{3,5}$ with $| \TT | \leq n$. 
Let $\TT \in \BT^{3,5}(k,p)$ with $| \TT | = n+1$, i.e. $n+1 = 3k + 5p +1$ for some $k$ and $p$.  Since
$\Psi_\phi(\TT) $ is  $|\TT^\infty|$-linear in $\phi$, we observe from \eqref{phi10} that the function $ \F [\Psi_\phi(\TT)]$ is supported on $ \{ \xi \in \R: | \xi | \leq 3 |\TT^\infty| (N+A) \} $. 
We divide our argument into two cases depending on the number of children of the root node.

\noi
{\it \underline{Case 1:}}
the root node $a$ of $\TT $ has three children.
We denote these three children by $a_s$, $s \in \{1,2,3\}$. 
By the notations in Remark \ref{REM:def}, let $ j_s := | \TT_{a_s} | =: 3k_s + 5p_s + 1  $ for $s \in \{1,2,3\}$. 
Then it follows that 

\noi
\begin{align}
k_1 + k_2 + k_3 + 1= k, \,\,p_1 + p_2 + p_3 = p
\label{KP1}
\end{align}
and
\[
\Psi_\phi(\TT)(t) =   -i
\int_0^t S(t-t') \Big( \Psi_{\phi}(\TT_{a_1})(t') \Psi_{\phi}(\TT_{a_2})(t')  \dx \overline{\Psi_{\phi}(\TT_{a_3})}(t') \Big) dt'.  \]
By unitarity of the linear propagator $S(t)$ in $\F L^1$, Young's inequality and the induction hypothesis, we have

\noi
\begin{align*}
\| \Psi_\phi(\TT)(t) \|_{\F L^1} & \leq 
\int_0^t  \big\| \Psi_{\phi}(\TT_{a_1})(t') \Psi_{\phi}(\TT_{a_2})(t') \overline{ \dx \Psi_{\phi}(\TT_{a_3})}(t') \big\|_{\F L^1} dt' \\
& \leq 3 \cdot |\TT^{\infty}_{a_3}| (N+ A) 
\int_0^t  \prod_{i=1}^3 \big\| \Psi_{\phi}(\TT_{a_i})(t') \big\|_{\F L^1} dt' \\
& \leq 3 \cdot |\TT^{\infty}_{a_3}| (N+ A) 
\int_0^t  \prod_{i=1}^3 (Ct)^{k_i+p_i} N^{k_i}  (RA)^{2k_i + 4p_i + 1} dt' \\
& = 3 \cdot |\TT^{\infty}_{a_3}| (N+ A) \int_0^t (Ct')^{k+p-1} N^{k-1}   (RA)^{2 (k-1) + 4p + 3} dt' \\
& \leq  6 \frac{|\TT^{\infty}|}{k+p}  C^{k+p-1} t^{k+p}  N^k (RA)^{2k+4p+1},
\end{align*}
where we used $ N \geq A $ in the last inequality. 
Then the desired estimate follows by choosing $C$ large enough and noting that $|\TT^{\infty}| = 2k + 4p +1 \le 5(k+p)$. 

\noi
{\it \underline{Case 2:}}
the root node $a$ has five children. We denote them by $a_s$ for $s \in \{1,2,3,4,5 \}$. 
Let $ j_s := | \TT_{a_s} | =: 3k_s + 5p_s + 1  $ for $s \in \{1,2,3,4,5\}$. 
We notice that 

\noi
\begin{align}
k_1 + k_2 + k_3  + k_4 + k_5= k, \,\, p_1 + p_2 + p_3 + p_4 + p_5 +1 = p,
\label{KP2}
\end{align}
and

\noi
\begin{align*}
\Psi_\phi(\TT)(t) & =   -i
\int_0^{t} S(t-t') \Big( \Psi_{\phi}(\TT_{a_1}) \overline{ \Psi_{\phi}(\TT_{a_2})}   \Psi_{\phi}(\TT_{a_3})  \overline{ \Psi_{\phi}(\TT_{a_4})}  \Psi_{\phi}(\TT_{a_5}) \Big) dt'.
\end{align*} 


\noi
As in the previous case, we bound

\noi
\begin{align*}
\| \Psi_\phi(\TT)(t) \|_{\F L^1} & \leq 
\int_0^t \big\|  \Psi_{\phi}(\TT_{a_1}) \overline{ \Psi_{\phi}(\TT_{a_2})}   \Psi_{\phi}(\TT_{a_3})  \overline{ \Psi_{\phi}(\TT_{a_4})}  \Psi_{\phi}(\TT_{a_5}) \big\|_{\F L^1} dt' \\
& \leq 
\int_0^t  \prod_{i=1}^5 \big\| \Psi_{\phi}(\TT_{a_i})(t') \big\|_{\F L^1} dt'\\
& \leq 
\int_0^t  \prod_{i=1}^5 (Ct)^{k_i+p_i} N^{k_i}  (RA)^{2k_i + 4p_i + 1} dt' \\
& \leq  \int_0^t (Ct')^{k+p-1} N^{k}   (RA)^{2 k + 4(p-1) + 5} dt' \\
& \leq    C^{k+p-1} t^{k+p} N^k (RA)^{2k+4p+1},
\end{align*}
which again gives \eqref{nonlin1a} provided $C \geq 1$. This finishes the proof of \eqref{nonlin1a}. The proofs of \eqref{nonlin1b} and \eqref{nonlin1c} follow from the similar arguments and we omit details.
\end{proof}

\noi
\begin{lemma}\label{LEM:nonlin2} Let
$\phi$ be as in \eqref{phi10}. Then we have for $t \geq 0$ and $s <0$,
\begin{align}
\big\| \Xi_{k,p}^{3,5} (\phi) (t) \big\|_{H^s} \leq C^{k+p} f_s(A)  t^{k+p} N^{k}(RA)^{2k+4p}R.
\label{nonlin2a}
\end{align}
	
\noi
Here, $f_s(A)$ is given by 
\begin{align}
f_s	(A) = 
\begin{cases}
1, & \text{if } s < -\frac 1 2, \\
(\log A)^\frac 12, & \text{if } s = -\frac 1 2, \\
A^{\frac{1}{2}+s} & \text{if } s > -\frac 1 2.
\end{cases}
\label{DEFf_s}
\end{align}
\end{lemma}

\noi
\begin{proof} 
Recall that $|\TT| = 3k+5p+1$ for $\TT \in \BT^{3,5}(k,p)$.
In view of Lemma \ref{LEM:tree}, \eqref{nonlin2a} follows from the bound

\noi
\begin{align}
\| \Psi_\phi (\TT)(t) \|_{H^s} \leq C^{k+p} f_s(A)  t^{k+p} N^{k}(RA)^{2k+4p}R,
\label{nonlin2}
\end{align}

\noi
for all $(k,p) \in (\N \cup \{0\})^2$ and $\TT \in \BT^{3,5}(k,p)$. 
We note that $\Psi_{\phi} (\TT)$ is $|\TT^\infty|$-linear in $\phi$,
and thus the support of $\F (\Psi_{\phi} (\TT))$ 
is contained in at most $2^{|\TT^\infty|}$
intervals of length $|\TT^{\infty}| \cdot A$. Furthermore, since
 $\jb \xi ^s$ is a decreasing function in $|\xi|$ for $s<0$, we thus have
 
\noi
\begin{align}
\label{decay}
\| \jb{\xi}^s \|_{L^2 (\supp \F[\Psi_\phi(\TT)  (t)])}  \le 
2^{\frac{|\TT^\infty|}2} f_s \big(|\TT^{\infty}|A\big),
\end{align}

\noi
uniformly in $t \ge 0$.

If $| \TT | = 1$, then the claimed result follows from \eqref{data}. We now assume that $|\TT| >1$ so that $k+p \ge 1$. Let us first assume that the root node of $\TT$ has three children $a_s$ for $s \in \{1,2,3\}$. 
Then by \eqref{decay}, Lemma \ref{LEM:nonlin1}, and \eqref{KP1}, we get  
\begin{align*}
 \big\| \Psi_{\phi}(\TT)(t) \big\|_{H^s}  & \leq  \| \jb{\xi}^s \|_{L^2 (\supp \F[\Psi_\phi(\TT)  (t)])}  \int_0^t \| \Psi_{\phi}(\TT_{a_1}) \Psi_{\phi}(\TT_{a_2})  \dx  \overline{\Psi_{\phi}(\TT_{a_3})}  \|_{\F L^{\infty}}    dt' \\
 & \le 2^{\frac{|\TT^\infty|}2} f_s \big(|\TT^{\infty}|A\big)  \int_0^t \|  \Psi_{\phi}(\TT_{a_1})  \|_{\F L^1} \| \Psi_{\phi}(\TT_{a_2}) \|_{\F L^1} \| \partial_x \Psi_{\phi}(\TT_{a_3})  \|_{\F L^{\infty}}    dt' \\
 & \le 2^{\frac{|\TT^\infty|}2} f_s \big(|\TT^{\infty}|A\big)   \int_0^t \prod_{i=1}^2 (Ct)^{k_i+p_i} N^{k_i}  (RA)^{2k_i + 4p_i + 1} \\
 & \hphantom{XXXXXXXXXX} \times (Ct)^{k_3+p_3} N^{k_3+1}  (RA)^{2k_3 + 4p_3}
 R   dt' \\
 & \le 2^{\frac{|\TT^\infty|}2} \frac{|\TT^{\infty}|}{k + p} f_s(A) (Ct)^{k+p} N^{k} (RA)^{2k+4p}R \\
 & \le C^{k+p} f_s(A)  t^{k+p}N^k (RA)^{2k+4p}R,
\end{align*}

\noi
for $C > 0$ large enough. Note that in the last inequality we used $|\TT^{\infty}| = 2k+4p + 1 \le 5 (k+p)$.

We then consider the case when the root node of $\TT$ has 
five children $a_s$ for $s \in \{1,2,3,4,5\}$. 
Then by \eqref{decay}, Lemma \ref{LEM:nonlin1}, and \eqref{KP2}, we get  
\begin{align*}
 \big\| \Psi_{\phi}(\TT)(t) \big\|_{H^s}  & \leq  \| \jb{\xi}^s \|_{L^2 (\supp \F[\Psi_\phi(\TT)  (t)])}  
 \int_0^t \Big\|\prod_{i=1}^5 \Psi_{\phi}(\TT_{a_i})  \Big\|_{\F L^{\infty}}    dt' \\
 & \le 2^{\frac{|\TT^\infty|}2} f_s \big(|\TT^{\infty}|A\big) 
 \int_0^t \prod_{i=1}^4 \|  \Psi_{\phi}(\TT_{a_i})  \|_{\F L^1}  \| \Psi_{\phi}(\TT_{a_5})  \|_{\F L^{\infty}}    dt' \\
 & \le 2^{\frac{|\TT^\infty|}2} f_s \big(|\TT^{\infty}|A\big)   \int_0^t \prod_{i=1}^4 (Ct)^{k_i+p_i} N^{k_i}  (RA)^{2k_i + 4p_i + 1} \\
 & \hphantom{XXXXXXXXXX} \times (Ct)^{k_5+p_5} N^{k_5}  
 (RA)^{2k_5 + 4p_5}
 R   dt' \\
 & \le 2^{\frac{|\TT^\infty|}2}  \frac{|\TT^{\infty}|}{k + p} f_s(A) (Ct)^{k+p} N^{k} (RA)^{2k+4p}R \\
 & \le C^{k+p} f_s(A)  t^{k+p}N^k (RA)^{2k+4p}R,
\end{align*}

\noi
for $C>0$ large enough. This shows \eqref{nonlin2} and finishes the proof.
\end{proof}

As a consequence of the last lemma, we have the 
following estimate on the power expansion \eqref{power2}. It essentially states that under some choice of parameters the main contribution to the Picard iterates \eqref{Picard_j} comes from the multilinear terms corresponding to the quintic nonlinearity (or trees for which parents nodes all have five children), i.e. $\{\Xi_{0,p} (\phi)\}_{p \ge 1}$.

\noi
\begin{lemma}\label{LEM:nonlin3} 
Let $\phi$ be as in \eqref{phi10}. 
If $R^2 A^2 \gg N $, 
then for any $j \geq 1$ we have for $t\geq0$,
\begin{align}
\| \Xi_j(\phi)(t)  \|_{H^s} \leq C^j
 f_s(A)    t^j(RA)^{4j}R,
\label{nonlin3}
\end{align}

\noi
where $\Xi_j$ is given in \eqref{Picard_j}.
\end{lemma}

\noi
\begin{proof} Fix $j \geq 1$ and $t \geq 0$. 
By  \eqref{Picard_j}, it suffices to show
\begin{equation}
\| \Xi_{k,p}^{3,5} (\phi) (t) \|_{H^s}  \leq C^j
 f_s(A)  t^j(RA)^{4j}R
\label{bd1}
\end{equation}
for any $\TT \in \BT^{3,5}(k,p)$ with $k+p =j$. 
By Lemma \ref{LEM:nonlin2},
it suffices to show
\[ N^{j-p} (RA)^{2(j+p) } \ll (RA)^{4j}, \]
which is a consequence of the assumption $R^2 A^2 \gg N$.
\end{proof}

The following lemma shows that the multilinear expressions
$\Xi_j$ defined in \eqref{Picard_j}
are stable under suitable perturbations.

\begin{lemma}\label{LEM:nonlin4}
Let $\phi$ be as in \eqref{phi10} and $\psi \in \F^{-1} C^{\infty}_0(\R)$ with 
$\| \psi\|_{\F L^1} \les RA$ and
$\supp(\ft \psi) \subset [-M,M]$ for some $M\geq 0$.  
We further assume $ R^2 A ^2 \gg N$ and $N \gg M$. Then, there exists $C >0$ such that 
\begin{align}
\| \Xi_j (\phi + \psi)(t)
- \Xi_j (\phi)(t)\|_{L^2} \leq
 C^j \| \psi\|_{L^2}  (tR^4 A^4)^j
 \label{diff_esti}
\end{align}
for all $j \in \N$.
\end{lemma}

\begin{proof}
From  \eqref{SumTree} and \eqref{Picard_j}, we have
\begin{align}
 \Xi_j ( \phi + \psi) - \Xi_j (\phi)
& =  \sum_{\substack{ \TT \in \BT^{3,5}(k,p) \\ j=k+p }} \big( \Psi_{\phi+ \psi} (\TT) - \Psi_\phi (\TT)\big)  \notag \\
& =  \sum_{\substack{ \TT \in \BT^{3,5}(k,p) \\ j=k+p }} 
\sum_{\phi_i \in \{\phi, \psi\}}
 \Psi (\TT; \phi_1, \dots, \phi_{2k+4p+1}),
\label{nonlin22}
\end{align}

\noi
where
the second summation in $\phi_1, \dots, \phi_{2k+4p+1}$
is over all possible combinations 
of $\phi_i \in \{\phi, \psi\}$
with at least one occurrence of $\psi$.
Given $\TT \in \BT^{3,5}(k,p)$ with $(k,p) \in (\N\cup \{0\})^2$, 
we have $ \supp \F[\Psi (\TT; \phi_1, \dots, \phi_{2k + 4p +1})] \subset \{ \xi \in \R: | \xi | \leq 6 (2k+4p+1) N \} $ provided $A, M \ll N$.
Without loss of generality, we assume $\phi_{1} = \psi$.
A similar induction argument as in Lemma \ref{LEM:nonlin1}
and the fact $\| \psi\|_{\F L^1} \les RA$ yields

\noi
\begin{align*}
\begin{split}
\| \Psi (\TT; \psi, \phi_2, \dots, \phi_{2k+4p+1})(t)\|_{L^2}
& \leq  C^j t^{j}
 \| \psi\|_{L^2} N^k
 \prod_{\substack{j = 2}}^{2k+4p+1} \| \phi_j\|_{\F L^1} \\
& \leq  C^j t^j  \| \psi\|_{L^2} N^k \cdot (RA) ^{2k+4p} \\
&  \leq C^j t^j  \| \psi\|_{L^2} \cdot (RA) ^{4k+4p} ,
\end{split}
\end{align*}

\noi
provided $R^2A^2 \gg N$.
\end{proof}

To conclude this subsection, 
we remark that the series
\[
v_1 (t) = \sum_{j=0}^\infty \Xi_j (\phi + \psi) (t)
\] 

\noi
converges on $[-T,T]$ under the conditions of Lemma \ref{LEM:nonlin4}, as long as $\|\psi\|_{\F L^1} \les RA$ and $T \ll (RA)^{-4}
\ll N^{-2} \ll 1$.
In particular, the series \eqref{power2} converges as long as $T \ll (RA)^{-4} \ll N^{-2} \ll 1$. Let us note that the function $v_1$ solves the initial value problem \eqref{GdNLS} and \eqref{Duhamel} with initial data $\phi$ replaced by $\phi+ \psi$.

\subsection{First Picard iterate}
In this subsection, we obtain a lower bound on the second Picard iterate 
($\Xi_1 (\phi)$ in \eqref{Picard_j}). This will allow us to prove later that this term represents the main contribution to the series expansion \eqref{power2} and to the norm inflation phenomena. See Subsection \ref{SUBSEC:3.2}.
To this purpose, we first recall the following elementary bounds on characteristic functions of intervals.

\begin{lemma}\label{LEM:conv}
For any $a,b,c,d,e,\xi \in \ft \M$ and $A \ge 1$, there exists $c >0$ such that 
\begin{align*}
\ind_{a + Q_A}* \ind_{b + Q_A} * \ind_{c + Q_A} * \ind_{d + Q_A} * \ind_{e + Q_A} (\xi)\geq cA^4 \ind_{a+b+c+d +e+Q_A}(\xi).
\end{align*}
\end{lemma}

We now state the lower bound on $\Xi_1 (\phi)$.

\noi
\begin{proposition}\label{PROP:lw_bd}
Let $\phi$  be as in \eqref{phi10}.
Then, 
for $0 < t \ll N^{-2}$ and $R^2 A^2 \gg N$, 
we have
\begin{align}
\|  \Xi_1 ( \phi)(t)\|_{H^s} \ges f_s(A) \cdot t  R^5 A^4,
\label{lw_bd}
\end{align}
\noi
where $f_s(A)$ is the function defined in \eqref{DEFf_s}.
\end{proposition}

\noi	
\begin{proof}
From \eqref{SumTree}, we have,
\begin{align}
\Xi_1(\phi)(t)
: = \Xi_{1,0}^{3,5} (\phi)(t) +  \Xi_{0,1}^{3,5} (\phi)(t) = \Psi_{\phi}(\TT_3)(T) + \Psi_{\phi}(\TT_5)(T),
\label{lw_bd1}
\end{align}
where $\TT_3$ (resp. $\TT_5$) is the tree with one parent node and three (resp. five) descendents. Namely, $\Psi_{\phi}(\TT_3)(t) = \mathcal{J}^3[S(t)\phi]$ and $\Psi_{\phi}(\TT_5)(t) = \mathcal{K}^5[S(t)\phi]$, see \eqref{ternary_quinary_op}. From Lemma \ref{LEM:nonlin2}, we have

\noi
\begin{align}
\| \Psi_{\phi}(\TT_3)(t)\|_{H^s} \les f_s(A) \cdot tN (RA)^2 R .
\label{lw_bd2}
\end{align}

\noi
We now turn to the quintic term $\Psi_{\phi}(\TT_5)$
and have

\noi
\begin{align} 
\label{lw_bd3}
\begin{split}
\F\big[\Psi_{\phi}(\TT_5)(t)\big](\xi)
&  = - \frac{1}{2} e^{-i  |\xi|^2 t }
\intt_{\xi = \xi_1 - \xi_2 + \cdots + \xi_5 } \\
& \hphantom{XXX} \times \int_0^t  e^{ i t' (|\xi|^2 - |\xi_1|^2 + |\xi_2|^2 - |\xi_3|^2 + |\xi_4|^2 - |\xi_5|^2)  }dt'
 \\
& \hphantom{XXX} \times  \ft \phi(\xi_1) \overline{\ft \phi}(\xi_2) \ft \phi(\xi_3 )  \overline{\ft \phi}(\xi_4)
\ft \phi(\xi_5)  
d\xi_1 d\xi_2d\xi_3d\xi_4 d\xi_5.
\end{split}
\end{align}

\noi
From \eqref{phi10}, we have $|\xi_j| \les N$ for 
$\xi_j \in \supp \ft \phi$. Then, since $\xi = \xi_1 - \xi_2 + \xi_3-\xi_4 + \xi_5 $ we have

\noi
\begin{align*}
\big||\xi_1|^2 - |\xi_2|^2 + |\xi_3|^2 - |\xi_4|^2 + |\xi_5|^2 + |\xi|^2 \big| \les N^2,
\end{align*}

\noi
which implies that,
\noi
\begin{align}
\Re  \Big( e^{ i t' (|\xi_1|^2 - |\xi_2|^2 + |\xi_3|^2 - |\xi_4|^2 + |\xi_5|^2 + |\xi|^2)  } 
\Big) 
\geq \frac 12
\label{lw_bd4}
\end{align}

\noi
holds for all $0< t' \ll N^{-2}$.
Thus, by
\eqref{lw_bd3},
\eqref{lw_bd4}, 
and 
Lemma \ref{LEM:conv},
we arrive at 

\noi
\begin{align}
|\F\big[ \Psi_{\phi}(\TT_5)(t)\big](\xi)|
\ges t R^5 A^4 \cdot \ind_{Q_A}(\xi).
\label{lw_bd5}
\end{align}
Noting that
$\| \jb{\xi}^s\|_{L^2_\xi( Q_A)} 
\sim f_s(A)$,
we obtain from \eqref{lw_bd5},
\begin{equation}
\| \Psi_{\phi}(\TT_5)(t) \|_{H^s} \ges f_s(A) \cdot t (R A)^4 R.
\label{lw_bd6}
\end{equation}
Finally, by collecting \eqref{lw_bd1}, \eqref{lw_bd2}, \eqref{lw_bd6} along with the the assumption $R^2 A ^2 \gg N$, we obtain \eqref{lw_bd}.
\end{proof}


\section{Proof of Theorem \ref{THM:ill1}}

In this section, we present the proof of Theorem \ref{THM:ill1}. A density argument reduces the proof of Theorem \ref{THM:ill1} to the following statement.

\noi
\begin{proposition}\label{PROP:main-dNLS}
Let $ s < 0 $. Fix $\psi \in \mathcal S (\M)$ such that $\ft \psi  \in C^{\infty}_0(\R)$.
Then, 
given any $n \in \N$, 
there exist a 
solution $v_n$ to \eqref{GdNLS} 
and $t_n  \in \big(0, \frac 1n\big) $ such that 
\begin{align}
 \| v_n(0) - \psi \|_{H^s(\M)} < \tfrac 1n \qquad \text{ and } 
\qquad \| v_n(t_n)\|_{H^s(\M)} > n.
\label{main2}
\end{align}
\end{proposition}


In what follows, we prove Proposition \ref{PROP:main-dNLS}. Given $n\in \N$, let $N_n, A_n, R_n, T_n \gg 1$ to be chosen later. We omit the dependence in $n$ of these constants from now on for convenience.
We set $v_{0} = \psi + \phi$ with $\phi$ as in \eqref{phi10}.

Denote by $v = v_n$ the global solution to \eqref{GdNLS} with initial data $v (0) = v_{0}$.


\subsection{Proof of Proposition \ref{PROP:main-dNLS}}
\label{SUBSEC:3.2}
We now prove Proposition \ref{PROP:main-dNLS} 
We claim that it suffices to show that the following properties hold:
\begin{align*}
 \textup{(i)} & \quad N^s R A^\frac{1}{2} \ll \frac 1n, \\ 
 \textup{(ii)} & \quad T R^4 A^4 \ll 1,  \\ 
 \textup{(iii)} & \quad  n \ll f_s(A) \cdot T R^5 A^4,\\ 
 \textup{(iv)} & \quad  N \ll R^2 A^2, \\
 \textup{(v)} & \quad A \ll N,\\
 \textup{(vi)} & \quad T \ll N^{-2},
 \end{align*}
 
\noi
for some $A, R, T$, and $N \gg 1$, depending on $n$ and $\psi$. 

We first prove the above claim, i.e. we show how conditions (i)-(vi) 
imply Proposition \ref{PROP:main-dNLS}. 
By Lemma \ref{LEM:nonlin3} and Lemma \ref{LEM:nonlin4},
we have the following series expansion representation of solution $v$
\begin{align}
\label{power3}
v (t) : = \sum_{j = 0}^\infty \Xi_j (v_{0})(t), 
\end{align}

\noi
provided conditions (ii) and (v) hold.
By \eqref{data}, the first condition (i) ensures that the first inequality in \eqref{main2} holds.
By \eqref{power3}, \eqref{data}, Proposition \ref{PROP:lw_bd}, Lemma \ref{LEM:nonlin3}, and Lemma \ref{LEM:nonlin4} (which hold under (iv), (v) and (iv)), we have
\begin{align*}
 \| v (T) \|_{H^s} & \geq  \| \Xi_1 ( \phi)(T) \|_{H^s} - 
 \| \Xi_0 ( \phi + \psi )(T) \|_{H^s} \\
 & \hphantom{XX} -  \sum_{j=1}^{\infty} \| \Xi_1 (\phi)(T)
- \Xi_1 (\psi + \phi)(T)\|_{H^s}  - \sum_{j = 2}^\infty  \|
 \Xi_j ( \phi )(T) \|_{H^s}   \\
 & \ges f_s(A) \cdot TR^5 A^4   - (1+ N^sRA^{\frac12})  -  TR^4 A^4  -   f_s(A) \cdot T^2R^9A^8 \\
 & \sim TR^5 A^4 \cdot f_s(A) \ge n,
\end{align*} 

\noi
where we used (i), (ii) and (iii) in the last inequality.
Finally, choosing $N$ sufficiently large such that 
$R^4 A^4 \ge n$, together with (ii), imply $t_n = T \in (0,\frac1n)$.
This proves Proposition \ref{PROP:main-dNLS}. 

\medskip

It thus remains to 
verify the conditions (i)-(v).
In what follows, 
we consider the following three cases:

\medskip

\noi
$\bullet$
{\bf Case 1:}  $s <  -\frac 12$.
We set
\begin{align*}
A= N^{\frac \dl 5}, 
\quad 
R = N^{\frac 12}, 
\quad \text{and} 
\quad
T= N^{ - 2 - \delta}, 
\end{align*}
with $ \dl >0 $ sufficiently small such that $s + \frac12 + \frac\delta{10} <0$. 
The conditions \textup{(ii)}, (iv),  \textup{(v)} and (vi) are trivially satisfied for 
$N \gg 1$.
By choosing $N$ large enough, 
we have
\begin{align*}
 N^s R A^{\frac12} & = N^{ s + \frac{1}{2} + \frac\dl{10} } \ll \frac 1n, \\
TR^5 A^{4} &= N^{\frac 12 - \frac\dl5} \gg n,
\end{align*}

\noi
which verify the conditions (i) and (iii) as $f_s(A) =1$ in this case.

\medskip

\noi
$\bullet$
{\bf Case 2:} $s =  - \frac12$.
In this case we have $f_s(A)= (\log A)^{\frac 12}$. 
Set

\noi
\begin{align*}
A = (\log N)^{\frac 32}, 
\quad 
R= \frac{ N^{\frac 12}}{\log N}, 
\quad \text{and} 
\quad
T = N^{ - 2 - \delta}, 
\end{align*}

\noi
with $ 0 < \dl \ll 1 $. 
The conditions \textup{(ii)}, (iv), \textup{(v)} and (vi) are trivially satisfied for 
$N \gg 1$.
By choosing $N$ large enough, 
we have

\noi
\begin{align*}
N^{-\frac 12} R A^{\frac12} & = (\log N)^{-\frac 14} \ll \frac 1n, \\
T R^5 A^{4} f_{s} (A) & \ges N^{\frac 14} \gg n .
\end{align*}

\noi
Thus, the conditions (i) and (iii) are satisfied.

\medskip

\noi
$\bullet$
{\bf Case 3:} $-\frac 12<s <0$. In this case, 
we have $f_s(A) = A^{s+ \frac 12}$. Choose 

\noi
\begin{align*}
A = N^{1+2s+\frac94 \dl}, 
\quad 
R = N^{- \frac 1 2 - 2s - \frac{17}8\dl}, 
\quad \text{and} 
\quad
T = N^{ - 2 - \delta}, 
\end{align*}

\noi
with $ 0 < \dl \ll 1 $ satisfying $2s + \frac94 \dl < 0$
and $2s^2 - \frac{3\dl}2 + \frac94 \dl s >0$. 
We then have
\noi
\begin{align*}
N^s R A^{\frac12} & = N^{- \dl } \ll \frac 1n, \\
 TR^4 A^4 & = N^{-\frac{\dl}{2}}  \ll 1, \\
 T R^5 A^{s+\frac 9 2} & = N^{2s^2 - \frac{3\dl}2 + \frac94\dl s} \gg n, \\
 R^2 A^2 &= N^{1+ \frac \dl 4 } \gg N,\\
 A& =  N^{1+2s+\frac94 \dl} \ll N,
\end{align*}

\noi
provided $N$ large enough.
This verifies (i) - (vi).

\noi
\begin{remark}
This framework fails to provide the norm inflation at the endpoint regularity $s=0$, which concurs with Remark \ref{REM:sharp}. 
As a matter of fact, 
if $s=0$, 
the condition \textup{(i)} writes $ R \ll A^{- \frac 12} $, 
which implies $R^2 A^2 \ll A$. 
This is incompatible with conditions \textup{(iv)} and \textup{(v)}.
\end{remark}

\begin{ackno}\rm
Y.W.~ and Y.Z.~would like to thank Tadahiro Oh for suggesting this problem. Y.Z.~was supported by the 
 European Research Council (grant no.~864138 ``SingStochDispDyn"). Y.W. was supported by supported by the EPSRC New Investigator Award (grant no. EP/V003178/1).
\end{ackno}

\end{document}